\documentclass[11pt]{amsart} \textwidth=14.5cm \oddsidemargin=1cm
\usepackage[rgb]{xcolor}
\usepackage{tikz}
\usepackage{verbatim}
\usepackage{amsmath}
\usepackage{amsxtra}
\usepackage{amscd}
\usepackage{amsthm}
\usepackage{amsfonts}
\usepackage{amssymb}
\usepackage{eucal}
\usetikzlibrary{arrows}

\usepackage{latexsym,todonotes}

\usepackage[linktocpage=true]{hyperref}
\hypersetup{colorlinks,linkcolor=blue,urlcolor=orange,citecolor=blue}
\usepackage[all, cmtip]{xy}

\usepackage{stmaryrd}
\usepackage{color} 

\newtheorem{thm}{Theorem} [section]

\newtheorem{lem}[thm]{Lemma}
\newtheorem{prop}[thm]{Proposition}

\theoremstyle{definition}
\newtheorem{definition}[thm]{Definition}

\theoremstyle{remark}
\newtheorem{rem}[thm]{Remark}

\numberwithin{equation}{section}
\begin{document}

\newcommand{\thmref}[1]{Theorem~\ref{#1}}
\newcommand{\secref}[1]{Section~\ref{#1}}
\newcommand{\lemref}[1]{Lemma~\ref{#1}}
\newcommand{\propref}[1]{Proposition~\ref{#1}}
\newcommand{\corref}[1]{Corollary~\ref{#1}}
\newcommand{\remref}[1]{Remark~\ref{#1}}
\newcommand{\eqnref}[1]{(\ref{#1})}

\newcommand{\exref}[1]{Example~\ref{#1}}

\newtheorem{innercustomthm}{{\bf Theorem}}
\newenvironment{customthm}[1]
  {\renewcommand\theinnercustomthm{#1}\innercustomthm}
  {\endinnercustomthm}
  
  \newtheorem{innercustomcor}{{\bf Corollary}}
\newenvironment{customcor}[1]
  {\renewcommand\theinnercustomcor{#1}\innercustomcor}
  {\endinnercustomthm}
  
  \newtheorem{innercustomprop}{{\bf Proposition}}
\newenvironment{customprop}[1]
  {\renewcommand\theinnercustomprop{#1}\innercustomprop}
  {\endinnercustomthm}

\newcommand{\bbinom}[2]{\begin{bmatrix}#1 \\ #2\end{bmatrix}}
\newcommand{\cbinom}[2]{\set{\^!\^!\^!\begin{array}{c} #1 \\ #2\end{array}\^!\^!\^!}}
\newcommand{\abinom}[2]{\ang{\^!\^!\^!\begin{array}{c} #1 \\ #2\end{array}\^!\^!\^!}}
\newcommand{\qfact}[1]{[#1]^^!}

\newcommand{\nc}{\newcommand}
 \nc{\A}{\mathcal A} 
\nc{\Ainv}{\A^{\rm inv}}
\nc{\aA}{{}_\A}
\nc{\aAp}{{}_\A'}
\nc{\aff}{{}_\A\f}
\nc{\aL}{{}_\A L}
\nc{\aM}{{}_\A M}
\nc{\Bin}{B_i^{(n)}}
\nc{\dL}{{}^\omega L}
\nc{\Z}{{\mathbb Z}}
 \nc{\C}{{\mathbb C}}
 \nc{\N}{{\mathbb N}}
 \nc{\fZ}{{\mf Z}}
 \nc{\F}{{\mf F}}
 \nc{\Q}{\mathbb{Q}}
 \nc{\la}{\lambda}
 \nc{\ep}{\epsilon}
 \nc{\h}{\mathfrak h}
 \nc{\He}{\bold{H}}
 \nc{\htt}{\text{tr }}
 \nc{\n}{\mf n}
 \nc{\g}{{\mathfrak g}}
 \nc{\DG}{\widetilde{\mathfrak g}}
 \nc{\SG}{\breve{\mathfrak g}}
 \nc{\is}{{\mathbf i}}
 \nc{\V}{\mf V}
 \nc{\bi}{\bibitem}
 \nc{\E}{\mc E}
 \nc{\ba}{\tilde{\pa}}
 \nc{\half}{\frac{1}{2}}
 \nc{\hgt}{\text{ht}}
 \nc{\ka}{\kappa}
 \nc{\mc}{\mathcal}
 \nc{\mf}{\mathfrak} 
 \nc{\hf}{\frac{1}{2}}
\nc{\ov}{\overline}
\nc{\ul}{\underline}
\nc{\I}{\mathbb{I}}
\nc{\xx}{{\mf x}}
\nc{\id}{\text{id}}
\nc{\one}{\bold{1}}

\nc{\ua}{\mf{u}}
\nc{\nb}{u}
\nc{\inv}{\theta}
\nc{\mA}{\mathcal{A}}
\newcommand{\TT}{\mathbf T}
\newcommand{\TA}{{}_\A{\TT}}
\newcommand{\tK}{\widetilde{K}}
\newcommand{\al}{\alpha}
\newcommand{\Fr}{\bold{Fr}}

\nc{\Qq}{\Q(v)}
\nc{\U}{\bold{U}}
\nc{\uu}{\mathfrak{u}}
\nc{\Udot}{\dot{\U}}
\nc{\udot}{\dot{\mathfrak{u}}}
\nc{\f}{\bold{f}}
\nc{\fprime}{\bold{'f}}
\nc{\B}{\bold{B}}
\nc{\Bdot}{\dot{\B}}
\nc{\Dupsilon}{\Upsilon^{\vartriangle}}
\newcommand{\T}{\texttt T}
\newcommand{\vs}{\varsigma}
\newcommand{\Pa}{{\bf{P}}}
\newcommand{\Padot}{\dot{\bf{P}}}

\nc{\ipsi}{\psi_{\imath}}
\nc{\Ui}{{\bold{U}^{\imath}}}
\nc{\uidot}{\dot{\mathfrak{u}}^{\imath}}
\nc{\Uidot}{\dot{\bold{U}}^{\imath}}
 \nc{\be}{e}
 \nc{\bff}{f}
 \nc{\bk}{k}
 \nc{\bt}{t}
 \nc{\BLambda}{{\Lambda_{\inv}}}
\nc{\Ktilde}{\widetilde{K}}
\nc{\bktilde}{\widetilde{k}}
\nc{\Yi}{Y^{w_0}}
\nc{\bunlambda}{\Lambda^\imath}
\newcommand{\Iwhite}{\I_{\circ}}
\nc{\ile}{\le_\imath}
\nc{\il}{<_{\imath}}

\newcommand{\ff}{B}


\nc{\etab}{\eta^{\bullet}}
\newcommand{\Iblack}{\I_{\bullet}}
\newcommand{\wb}{w_\bullet}
\newcommand{\UIblack}{\U_{\Iblack}}

\newcommand{\blue}[1]{{\color{blue}#1}}
\newcommand{\red}[1]{{\color{red}#1}}
\newcommand{\green}[1]{{\color{green}#1}}
\newcommand{\white}[1]{{\color{white}#1}}

\newcommand{\huanchentodo}{\todo[inline,color=orange!20, caption={}]}
\newcommand{\wtodo}{\todo[inline,color=green!20, caption={}]}

\newcommand{\dvd}[1]{t_{\odd}^{{(#1)}}}
\newcommand{\dvp}[1]{t_{\ev}^{{(#1)}}}
\newcommand{\ev}{\mathrm{ev}}
\newcommand{\odd}{\mathrm{odd}}

 \author[Huanchen Bao]{Huanchen Bao}
\address{Department of Mathematics, National University of Singapore, Singapore 119076, Singapore.}
\email{huanchen@nus.edu.sg}

 \author[Thomas sale]{Thomas sale}
\address{Department of Mathematics, University of Virginia, Charlottesville, Virginia 22904, United states.}
\email{tws2mb@virginia.edu}

\title[Quantum symmetric pairs at roots of $1$]
{Quantum symmetric pairs at roots of $1$}

\begin{abstract}
	A quantum symmetric pair is a quantization of the symmetric pair of universal enveloping algebras. Recent development suggests that most of the theory for quantum groups can be generalised to the setting of quantum symmetric pairs. In this paper, we study the $\imath$quantum group at roots of $1$. We generalize Lusztig's quantum Frobenius morphism in this new setting. We define the small $\imath$quantum group and compute its dimension. 
\end{abstract}

\maketitle

\let\thefootnote\relax\footnotetext{{\em 2010 Mathematics Subject Classification.} Primary 17B10.}
\setcounter{tocdepth}{1}
\tableofcontents

\section{Introduction}

\subsection{} 
Let $\U$ be the quantum group of a semisimple complex Lie algebra $\mathfrak{g}$ with a parameter $v$. When $v$ is a root of $1$, the algebra $\U$ is closely related with the modular representation theory of algebraic groups and the representation theory of affine Lie algebras; cf. \cite[\S0.6]{Lu90a} and \cite{AJS94}.


\subsection{}

Let $\mA = \Z[v,v^{-1}]$  with a generic parameter $v$. We denote by $_{\mA}\U$ the $\mA$-form of $\U$, which is an $\mA$-subalgebra of $\U$ generated by the divided powers $E_{i}^{(n)}$ and $F_{i}^{(n)}$ for various  simple root $i$ and $n \in \Z_{\ge 0}$. 

Let $l$ be an odd positive integer, and prime to $3$ if $\U$ has a component of type $G_2$. Let $\mA'$ be the quotient of $\mA$ by the two-sided ideal generated by the $l$-th cyclotomic polynomial, that is, $v$ is an $l$-th root of $1$ in $\mA'$. Let ${}_{\mA'} \U = \mA' \otimes_{\mA} {}_{\mA}\U $.

Let ${}_{\mA'}\U^*$ be the universal enveloping algebra (over $\mA'$) of the Lie algebra $\mathfrak{g}$ (we refer to \S\ref{subsub:Ustar} for a more precise definition of the algebra ${}_{\mA'}\U^*$ in terms of root data). 

Lusztig  \cite[Theorem~8.10]{Lu90b}  defined the quantum Frobenius morphism
\begin{align}
\label{eq:LuFr} 
\Fr: {}_{\mA'}\U &\longrightarrow {}_{\mA'}\U ^{*}, \\
\notag E_{i}^{(n)} &\mapsto   
	\begin{cases}
	E_{i}^{(n/l)},  &\mbox{for $n \in l\mathbb{Z}$}; \\  
	0,  &\mbox{otherwise};
	\end{cases}
\\
	\notag	F_{i}^{(n)}  &\mapsto  
	\begin{cases}
	F_{i}^{(n/l)} ,  &\mbox{for $n \in l\mathbb{Z}$}; \\ 
	0,  &\mbox{otherwise}.
	\end{cases}
\end{align}

Lusztig also defined the small quantum group $ {\mathfrak{u}}$ as an $\mA'$-subalgebra of ${}_{\mA'}\U$ generated by $F_i^{(a)}  $, $E_i^{(a)}  $ with $0 \le a < l$. The algebra $ {\mathfrak{u}}$ should be viewed as the ``Frobenius kernel'' of the morphism \eqref{eq:LuFr} (cf. \cite[Appendix]{GK93}).

\subsection{}A quantum symmetric pair (QSP) $(\U, \Ui)$ is a quantization of the symmetric pair of enveloping algebras $(\U(\mathfrak{g}), \U(\mathfrak{g}^\theta))$, where $\theta : \mathfrak{g} \rightarrow \mathfrak{g}$ denotes an involution of Lie algebras. 
 They were originally developed by Letzter \cite{Le99} to study quantum homogeneous spaces and reflection equations in finite type, some of which was generalized to Kac-Moody type by Kolb \cite{Ko14}.  The algebra $\Ui$ is a coideal subalgebra of $\U$, which we often call the $\imath$quantum group.


Recent development in the theory of quantum symmetric pairs suggests that  many of the fundamental constructions and results for quantum groups can be generalised to the setting of quantum symmetric pairs. In \cite{BW18a, BW18b, BW18c}, the first author and Wang developed the theory of canonical bases  arising from  quantum symmetric pairs (called the $\imath$canonical bases). The integral form of the $\imath$quantum group is spanned by the $\imath$canonical basis over $\mA$, and generated by the $\imath$divided powers introduced in \cite{BW18c} as an $\mA$-algebra. The $\imath$divided powers are the generalization of the usual divided powers, but are far more complicated.

 \subsection{}Since $\Ui$ lacks suitable triangular decomposition, we generally consider the modified form $\Uidot$, together with the modified quantum group $\Udot$. The modified quantum group $\Udot$ (resp. $\Uidot$) is obtained by replacing the identity element in $\U$ (resp. $\Ui$) with a collection of idempotents. 
 
 We have the natural generalization of the quantum Frobenius morphism in the setting of $\Udot$, as well as the modified small quantum group $\dot{\mathfrak{u}}$ as established in \cite[\S35.1.9]{Lu}.


\subsection{}Let ${}_{\mA}\Uidot$ be the integral form of the modified $\imath$quantum group.
In this paper, we study the modified $\imath$quantum group ${}_{\mA'}\Uidot = \mA' \otimes_{\mA} {}_\mA\Uidot $ at the $l$-th root of $1$, where $l$ is a positive odd integer that is relatively prime to $3$ if $\U$ has a component of type $G_{2}$. We remark that while one can consider more general $l$ in the setting of quantum groups \cite[Chap.~35]{Lu}, it is much more restricted in the setting of $\imath$quantum groups (cf. Remark~\ref{rem:even}).

We prove that Lusztig's quantum Frobenius morphism restricts to an algebra homomorphism on modified $\imath$quantum groups in Theorem~\ref{thm:main}. We define the (modified) small $\imath$quantum group ${}_{\mA'}\dot{\mathfrak{u}}^\imath$ and compute its dimension in Theorem~\ref{thm:main2}. The $\imath$divided powers of the  $\mA$-algebra ${}_\mA\Uidot$ studied in \cite{BW18c, BeW18} play important roles in this paper. 

\subsection{}This paper is organised as follows. We recall quantum groups and quantum symmetric pairs in Section~\ref{sec:2}. We define and study the $\imath$quantum group $\U^{*,\imath}$ in Section~\ref{sec:3}.
We explicitly compute the restriction of the quantum Frobenius morphism on ${}_{\mA'}\Uidot$ in Section~\ref{sec:4}. 
We define the modified small $\imath$quantum group and compute its dimension in Section~\ref{sec:5}.

\vspace{.2cm}

{\bf Acknowledgement: }The authors would like to thank Weiqiang Wang for his helpful discussions and for encouraging the collaboration. HB is supported by a NUS start-up grant. TS was partially supported by Wang's NSF Grant DMS-1702254.

\section{Preliminaries}\label{sec:2}

In this section, we follow the conventions of \cite{Lu} and \cite{BW18b}.

\subsection{Cartan data and root data}
\subsubsection{}
Let $\langle Y, X, \dots \rangle$ be a root datum of finite type $(\I, \cdot)$; cf. \cite[\S 1.1.1]{Lu}. Let $\Phi \subset X$ (resp. $\Phi^+ \subset X$) be the set of roots (resp. positive roots). Let $\Phi^\vee \subset Y$ (resp. $\Phi^{\vee,+} \subset Y$) be the set of coroots (resp. positive coroots).

Let $v$ be an indeterminate and write $\mathcal{A} = \mathbb{Z}[v,v^{-1}].$ We define, for $n \in \mathbb{Z}$ and $a \in \mathbb{Z}_{\geq0}$, 
\[
[n] = \frac{v^{n} - v^{-n}}{v - v^{-1}}, \quad  [n]^{!} = [n][n-1]. . .[2][1], \quad \left[ \begin{array}{cc|r} n \\ a \end{array} \right] = \frac{[n]^{!}}{[n-a]^{!}[a]^{!}}.
\]
 For any $ i \in \I$, we define $v_{i} := v^{(i \cdot i)/2}$ and also 
\[
[n]_i = \frac{v_i^{n} - v_i^{-n}}{v_i - v_i^{-1}}, \quad  [n]_i^{!} = [n]_i[n-1]_i. . .[2]_i[1]_i, \quad \left[ \begin{array}{cc|r} n \\ a \end{array} \right]_i = \frac{[n]_i^{!}}{[n-a]_i^{!}[a]_i^{!}}.
\]

We also define
\[
[n]_{v^2_i} = \frac{v_i^{2n} - v_i^{-2n}}{v^2_i - v_i^{-2}}, \text{ and similarly define } \quad  [n]_{v^2_i}^{!}, \quad \left[ \begin{array}{cc|r} n \\ a \end{array} \right]_{v^2_i}.
\]
\subsubsection{}\label{subsec:stardatum}
Throughout this paper, we assume that $l > 0$ is an odd integer, and $l$ is prime to $3$ if the Cartan datum has a factor of $G_2$. We define a new Cartan datum $(\I, \circ)$ with the same $\I$ and the pairing 
\[
 i \circ j = (i \cdot j)l_{i}l_{j}, 
\]
where $l_i$ is the smallest positive integer such that $l_{i}(i \cdot i)/2 \in l\mathbb{Z}$; cf. \cite[\S2.2.4]{Lu}.

We  further define a new root datum $(Y^{*},X^{*}, \dots)$ of type $(\mathbb{I}, \circ)$ following  \cite[\S 2.2.5]{Lu}. Define $X^{*} = \{ \la \in X | \mbox{$\langle i, \la \rangle \in l_i\mathbb{Z}$ for all $i \in I$} \}$ and $Y^{*} = \mathrm{Hom}(X^{*}, \mathbb{Z})$ with the obvious bilinear pairing, which we denote by $\langle \cdot,\cdot \rangle^*$. The map $\I \rightarrow X^*$ is given by $i \mapsto i\rq{}^* = l_i i\rq{} \in X$. The map $\I \rightarrow Y^*$ associates to $i \in \I$ the element $i^* \in Y^*$ such that 
$\langle i^*, \zeta \rangle^* = \langle i, \zeta \rangle / l_i$ for any $\zeta \in X^*$.

We define $v_{i}^{*} :=v_{i}^{(i \circ i)/2}$. We also define $[n]^*_i$, $([n]_i^*)^{!}$,  $\left[ \begin{array}{cc|r} n \\ a \end{array} \right]_i^*$ in the obvious way, for $i \in \I$. 

\begin{lem}\label{lem:sametype}
	We have $	\langle i^*, j\rq{}^{*} \rangle ^* = \langle i, j\rq{} \rangle$ for any $i, j \in \I$.
\end{lem}

\begin{proof}
Recall $\langle i^*, j\rq{}^{*} \rangle ^* = \langle i, l_j j\rq{}\rangle/ l_i $. If $  \langle i, j\rq{}\rangle = 0$ or $j = i$, we trivially have $\langle i^*, j\rq{}^{*} \rangle ^*  = \langle i, j\rq{}\rangle$. Otherwise, we could have $\langle i, j\rq{}\rangle = 1, 2, 3$ thanks for our finite type assumption. Then since $l$ is odd and prime to $3$ if there is a $G_2$ factor, we must have $l_i = l_j$. The lemma follows.
\end{proof} 

As a consequence of the proof, we must have $l=l_i$ for all $i \in \I$ (cf. \cite[\S8.4]{Lu90b}).
\subsection{Quantum groups}

\subsubsection{}
Given a root datum $(Y, X,\dots)$ of type $(\mathbb{I}, \cdot)$, let $\U$ be the associated quatnum group over $\Qq$ generated by $E_{i}$, $F_{i}$, and $K_{\mu}$, for all $i \in I$ and $\mu \in Y$. We write  $E_{i}^{(n)} = E_{i}^{n}/[n]_{i}^{!}$ and $F_{i}^{(n)} = F_{i}^{n}/[n]_{i}^{!}$ for $n \in \mathbb{Z}_ {\ge 0}$.


Recall \cite[\S 23.1]{Lu} the modified version of $\U$, denoted by $\dot{\U}$. The algebra $\dot{\U}$ has the direct sum decomposition
\[
\dot{\U} = \bigoplus_{\la, \la' \in X}  {}_{\la}\U_{\la'}.
\]
The $\mathcal{A}$-form $_{\mathcal{A}}\dot{\U}$ is the $\mA$-subalgebra of  $\dot{\U}$  generated by:
\[
E_{i}^{(n)}1_{\la} = 1_{\la + n i'}E_{i}^{(n)} = 1_{\la + n i'}E_{i}^{(n)}1_{\la} \in {}_{\la + ni'}\U_{\la}
,\]
and
\[
F_{i}^{(n)}1_{\la} = 1_{\la - ni'}F_{i}^{(n)} = 1_{\la - ni'}F_{i}^{(n)}1_{\la} \in {}_{\la - ni'}\U_{\la}
,\]
for various $n \in \mathbb{Z}_{\geq 0}$, $i \in \I$, $\lambda \in X$.

We define the modified quantum group with scalars in a commutative $\mathcal{A}$-algebra $R$ as follows:
\[
 {}_{R}\dot{\U} := R \otimes_{\mathcal{A}} {}_\mathcal{A}\dot{\U}.
\]


\subsubsection{}\label{subsec:com}
We consider the completion ${\Udot}^\wedge$ of $\Udot$ as follows; cf. \cite[\S36.2.3]{Lu}. Recall any element in $\Udot$ can be written uniquely as a (finite) sum $\sum_{\lambda, \lambda'} x_{\lambda,\lambda'}$ for $x_{\lambda,\lambda'} \in {}_\lambda\U_{\lambda'}$. We consider infinite summation of the form 
\[
\sum_{\lambda, \lambda'} x_{\lambda,\lambda'},
\]
as long as there is a finite set $F \in X$ such that $x_{\lambda,\lambda'} =0$ unless $\lambda - \lambda' \in F$. The algebra structure of $\Udot$ extends naturally to an algebra structure of $\Udot^\wedge$. We similarly define ${}_\mA \Udot^\wedge$ and ${}_R \Udot^\wedge$ for any $\mA$-algebra $R$.

\subsubsection{}
\label{subsub:Ustar}
Let $\U^*$ be the quantum group over the field $\Qq$ associated with the root datum 
$(Y^{*},X^{*}, \dots)$ of type $(\mathbb{I}, \circ)$, generated by (by abuse of notations) $E_{i}$, $F_{i}$, and $K_{\mu}$, for all $i \in I$ and $\mu \in Y^*$. 

We abuse the notations and write $E_{i}^{(n)} := E_{i}^{n}/([n]_{i}^*)^{!} \in \U^*$ and $F_{i}^{(n)} := F_{i}^{n}/([n]_{i}^*)^{!} \in \U^*$. We similarly define $\dot{\U}^*$, ${}_{\mathcal{A}}\dot{\U}^*$, $_{R}\U^*$, ${}_{R}\dot{\U}^*$. 

We also define similarly  the completions ${\Udot}^{*,\wedge}$, ${}_\mA{\Udot}^{*,\wedge}$, as well as  ${}_R{\Udot}^{*,\wedge}$ for any $\mA$-algebra $R$.


\subsubsection{} Let $\mA\rq{}$ be the quotient of  $\mA$ by the two-sided ideal generated by the $l$-th cyclotomic polynomial $f_{l} \in \mA$. Recall that $(f_1, f_2, f_3, \dots) = (v-1, v+1, v^2+v+1, \dots)$. We denote by $\phi : \mA \rightarrow \mA\rq{}$ the quotient map. 

\begin{thm}\cite[\S35.1.9]{Lu}\label{thm:LuFr}
There is a homomorphism of $\mA\rq{}$-algebras $\bold{Fr}: {}_{\mA\rq{}}\dot{\U} \rightarrow {}_{\mA\rq{}}\dot{\U}^{*}$ such that
\[ 
\bold{Fr}: E_{i}^{(n)}1_{\la} \mapsto 
	\begin{cases}
		E_{i}^{(n/l)}1_{\la},&\text{for $n \in l\mathbb{Z}$, and $\la \in X^{*}$};\\
		0,  &\text{otherwise}.
	\end{cases}
	\]
and
\[ 
\bold{Fr}: F_{i}^{(n)}1_{\la} \mapsto 
	\begin{cases}
		F_{i}^{(n/l)}1_{\la}, &\mbox{for $n \in l\mathbb{Z}$, and $\la \in X^{*}$}; \\
		0,  &\mbox{otherwise}.
	\end{cases}
\]
\end{thm}
Note that $\bold{Fr}$ extends natually to an algebra homomorphism 
\[
\bold{Fr} : {}_{\mA\rq{}}\dot{\U}^\wedge \rightarrow {}_{\mA\rq{}}\dot{\U}^{*, \wedge}.
\]


\subsection{Quantum symmetric pairs}
\subsubsection{}
 Let $\tau$ be an involution of the Cartan datum $(\I, \cdot)$; we allow $\tau =1$. We further assume that $\tau$ extends to  an involution on $X$ and an involution on $Y$, respectively, such that the perfect bilinear pairing is invariant under the involution $\tau$. 

Let $\I_{\bullet} \subset \I$. We have a subroot datum $(X, Y, \dots)$ of type $(\I_\bullet, \cdot)$.  Let $W_{\I_\bullet}$ be the parabolic subgroup of $W$ with $w_{\bullet}$ as  its longest element.  Let $\Phi_\bullet \subset X$ (resp. $\Phi^+_\bullet \subset X$) be the set of roots (resp. positive roots). Let $\Phi^\vee_\bullet \subset Y$ (resp. $\Phi^{\vee,+}_\bullet \subset Y$) be the set of coroots (resp. positive coroots).

Let $\rho^\vee_{\bullet}$ be the half sum of all positive coroots in the set $\Phi^{\vee}_{\bullet}$, 
and let $\rho_{\bullet}$ be the half sum of all positive coroots in the set $\Phi _{\bullet}$. 
We shall write 
\begin{equation}
\label{eq:white}
 \I_{\circ} = \I \backslash \I_{\bullet}.
\end{equation}
%

A pair $(\I_{\bullet}, \tau)$ is called {\em admissible} (cf. \cite[Definition~2.3]{Ko14}) if the following conditions are satisfied (with respect to the root datum $(X,Y, \dots)$ of type $(\I, \cdot)$):
\begin{itemize}
	\item	[(1)]	$\tau (\I_{\bullet}) = \I_{\bullet}$; 
	\item	[(2)]	The action of $\tau$ on $\I_{\bullet}$ coincides with the action of $-w_{\bullet}$; 
	\item	[(3)]	If $j \in \I_{\circ}$ and $\tau(j) = j$, then $\langle \rho^\vee_{\bullet}, j' \rangle \in \Z$.
\end{itemize}

\subsubsection{}

Let
$ \inv =  -w_{\bullet} \circ \tau$
be an involution of $X$ and $Y$. Following \cite{BW18b}, we introduce the $\imath$-weight lattice and $\imath$-root lattice
\begin{align}
  \label{XY}
 \begin{split}
X_{{\imath}} = X /  \breve{X}, & \quad \text{ where } \; \breve{X}  = \{ \la - \inv(\la) \vert \la \in X\},
 \\
Y^{\imath} &= \{\mu \in Y \big \vert \inv(\mu) =\mu \}.
\end{split}
\end{align}
For any $\la \in X$, we denote its image in $X_{\imath}$ by $\overline{\la}$.

The involution $\tau$ of $\I$ induces an automorphism of the $\Q(v)$-algebra $\U$, denoted also by $\tau$,
under which $E_i \mapsto E_{\tau i}, F_i \mapsto F_{\tau i}$, and $K_\mu \mapsto K_{\tau \mu}$. 

\subsubsection{}
We recall here the definition of quantum symmetric pair $(\U, \Ui)$ following \cite[\S3.3]{BW18b}. 

\begin{definition}\label{def:Ui}
The algebra $\Ui$, with parameters 
\begin{equation}
  \label{parameters}
\vs_{i} \in \pm v^\Z, \quad \kappa_i \in \Z[v,v^{-1}], \qquad \text{ for }  i \in \I_{\circ},
\end{equation}
is the $\Qq$-subalgebra of $\U$ generated by the following elements:
\begin{align*}
F_{i}  &+ \vs_i \T_{w_{\bullet}} (E_{\tau i}) \widetilde{K}^{-1}_i 
+ \kappa_i \widetilde{K}^{-1}_{i} \, (i \in \I_{\circ}), 
 \\
& \quad K_{\mu} \,(\mu \in Y^{\imath}), \quad F_i \,(i \in \I_{\bullet}), \quad E_{i} \,(i \in \I_{\bullet}).
\end{align*}
The parameters are required to satisfy Conditions \eqref{kappa}-\eqref{vs2}: 
\begin{align}
 \label{kappa}
 \begin{split}
\kappa_i &=0 \; \text{ unless } \tau(i) =i, \langle i, j' \rangle = 0 \; \forall j \in \Iblack,
\\
&  
\qquad 
\text{ and } \langle k,i' \rangle \in 2\Z \; \forall k = \tau(k) \in \Iwhite \text{ such that } \langle k, j' \rangle = 0 \; \text{ for all } j \in \Iblack;
\end{split}
\\
\overline{\kappa_i} &= \kappa_i;   \label{kappa2}   
\\
\vs_{i} & =\vs_{{\tau i}} \text{ if }    i \cdot \theta (i) =0;
\label{vs=}
\\
\vs_{{\tau i}} &= (-1)^{ \langle 2\rho^\vee_{\bullet},  i' \rangle } v_i^{-\langle i, 2\rho_{\bullet}+\wb\tau i ' \rangle} {\ov{\vs_{i} }}.   \label{vs2}
\end{align}
\end{definition}

\begin{rem}
Note that we require $\vs_{i} \in \pm v^\Z$, instead of $\vs_{i} \in \Z[v,v^{-1}]$ as in \cite{BW18c}. Note that such $\vs_i$ always exists by \cite[Remark 3.14]{BK15}, thanks to the establishment of the \cite[Conjecture~2.7]{BK15} in \cite[Theorem~4.1]{BW18c}.
\end{rem}

%

We denote by $\Uidot$ the modified coideal subalgebra as defined in \cite[\S3.7]{BW18b}. It follows from \cite[\S3.7]{BW18b} that $\Udot$ is naturally a $\Uidot$-bimodule.


\subsubsection{}
We recall here the integral from ${}_\mA \Uidot$ of the modified coideal subalgebra $\Uidot$. 

	\begin{definition}\cite[Definition~3.10]{BW18c}
	We define ${}_\mA \Uidot$ to be the set of elements $u \in \Uidot$ such that $u \cdot m \in{}_\mA\Udot$ for all $m \in {}_\mA\Udot$.
	\end{definition}
\begin{prop}\cite[Theorem~5.1, Theorem~7.2, Corollary~7.5]{BW18c}\label{prop:idiv}
	\begin{enumerate}
		\item ${}_\mA \Uidot$ is a free $\mA$-module.
		\item ${}_\mA \Uidot$ is generated as an $\mA$-subalgebra of $\Uidot$ by the $\imath$-divided powers $B^{(a)}_{i, \zeta}$ ($ i \in \I$) and $E^{(a)}_{j} 1_{\zeta}$ ($j \in \Iblack$) for $\zeta \in X_\imath$ and $a \ge 0$. 
	\end{enumerate}
\end{prop}
The $\imath$-divided powers were introduced in \cite[Theorem~5.1]{BW18c} in order to construct the canonical basis for quantum symmetric kac-Moody pairs. They play a crucial role in this paper.

\subsubsection{}
We consider the completion ${\Udot}^{\imath,\wedge}$ of $\Uidot$ analogous to that of $\Udot$ in \S\ref{subsec:com}  by allowing infinite summation 
\[
\sum_{\lambda, \lambda'} x_{\lambda,\lambda'}
\]
as long as there is a finite set $G \in X_\imath$ such that $x_{\lambda,\lambda'} =0$ unless $\lambda - \lambda' \in G$. 
The following lemma is straightforward.
\begin{lem}\label{lem:emd}
We have the algebra embedding 
\begin{align*}
\imath: \Udot^{\imath,\wedge} &\longrightarrow {\Udot}^\wedge,\\
	 x &\mapsto  \sum_{\lambda \in X} x \one_{\lambda}.
\end{align*}
In particular, $\imath$ restricts to embeddings $\imath: \Uidot \rightarrow {\Udot}^\wedge$, and $\imath:  {}_\mA\Udot^{\imath} \rightarrow {}_\mA{\Udot}^\wedge$. 

\end{lem}

\subsubsection{}
For a commutative $\mA$-algebra $R$, we define the modified coideal subalgebra with scalars in $R$ as 
\[
 {}_R \Uidot = R \otimes_{\mA} {}_\mA \Uidot.
\]
We similarly define $ {}_R \Udot^{\imath,\wedge}$.

\begin{prop}\label{prop:emd}
Let $R$ be a commutative $\mA$-algebra. The following induced embedding after base change is injective 
\begin{equation}\label{eq:emb}
\imath:  {}_R \Udot^{\imath} \longrightarrow {}_R{\Udot}^\wedge.
\end{equation}
\end{prop}

\begin{proof}
Recall the canonical basis $\dot{\B}^\imath$ of ${}_\mA \Uidot$ from \cite[Theorem~7.2]{BW18c}. For any $b_1\diamondsuit^\imath_\zeta b_2 \in \dot{\B}^\imath$ and $\lambda \in X$ such that $\overline{\lambda} = \zeta \in X_\imath$, we have 
\[
	b_1\diamondsuit^\imath_\zeta b_2 1_\lambda = b_1\diamondsuit_\lambda b_2 1_\lambda + \text{lower terms}.
\]
Here $b_1\diamondsuit_\lambda b_2$ denotes Lusztig\rq{}s canonical basis element (\cite[Theorem~25.2.1]{Lu}) on ${}_\mA \Udot$. The proposition follows, since the coefficient of the leading term is $1$. 
\end{proof}


\section{The $\imath$quantum group $\U^{*,\imath}$}\label{sec:3}
In this section, we define the $\imath$quantum group associated with the root datum $(X^*, Y^*,\dots)$ and the pair $(\Iblack, \tau)$. 

\subsection{Root data} 
Recall the root datum $(X^*, Y^*, \dots)$ in \S\ref{subsec:stardatum}.   Note that since $\tau$ is an involution of the Cartan datum $(\I, \cdot)$, it is naturally an involution the Cartan datum $(\I, \circ)$. The involution $\tau$ restricts to an involution of $X^* \subset X$. The involution $\tau$ extends naturally on $Y^* = \text{Hom}(X^*, \mathbb{Z})$, such that the perfect pairing $\langle \cdot, \cdot \rangle^*$ is $\tau$-invariant. Thanks to Lemma~\ref{lem:sametype}, the pair $(\Iblack, \tau)$ is admissible with respect to the root datum $(X^*, Y^*, \dots)$.

Recall the definition of $X_\imath$ and $Y^\imath$ in \eqref{XY}. We similarly define $X^*_\imath$ and $Y^{*,\imath}$. In particular, we have 
\[
X^*_{{\imath}} = X^* /  \breve{X}^*,  \quad \text{ where } \; \breve{X}^*  = \{ \la - \inv(\la) \vert \la \in X^*\},
\]

\begin{lem}
If $\lambda \in X^*$, then we have $\theta (\lambda) \in X^*$.
\end{lem}

\begin{proof}
Note that $X^*$ is $\tau$-invariant by our assumption on $\tau$. The sublattice $X^*$ is invariant under the Weyl group action as well, thanks to Lemma~\ref{lem:sametype}.  Therefore $X^*$ is invariant under $\theta$. 
\end{proof}

So we have $\breve{X}^* \subset X^* \cap \breve{X}$ and the following commutative diagram: 
\[
\xymatrix{ X^* \ar[r] \ar[d] & X \ar[d] \\
X^*_\imath \ar[r]& X_\imath}
\]

\begin{lem}
If the root datum $(Y,X, \dots )$ is simply connected,  then we have  $\breve{X}^* = X^* \cap \breve{X}$.
\end{lem}

\begin{proof}
We need to show that if $\langle i, \lambda -\theta(\lambda) \rangle \in l\mathbb{Z}$, then $\lambda -\theta(\lambda) = \mu -\theta(\mu)$ for some $\mu \in X^*$. 

We write $X = \oplus_{i \in \I}\mathbb{Z} \omega_{i}$, where $\omega_i$ denotes the $i$-th fundamental weight. Let 
\[
	\lambda = \sum_{i \in \I } a_i \omega_i.
\]
Note that since 
\[
\theta(\omega_i) = 
	\begin{cases}
		-\omega_{\tau i}, &\text{if } i \in \Iwhite;\\
		\omega_{i}, &\text{if } i \in \Iblack,
	\end{cases}
\]
we have 
\[
	\lambda - \theta(\lambda) = \sum_{i \in \Iwhite} (a_i + a_{\tau i}) \omega_i.
\]

Since $\langle i, \lambda -\theta(\lambda) \rangle \in l\mathbb{Z}$, we have
\[
	  a_i \equiv - a_{\tau i},   \mod l .
\]
Since $l$ is odd, then we must have 
\[
 a_i \equiv  a_{\tau i} \equiv 0,  \mod l, \text{ if }  i = \tau i.
 \]
 Therefore by considering $\tau$-orbits in $\Iwhite$, we can find $1-l \le b_i \le l-1$ for each $i \in \Iwhite$ such that 
 \[
 b_i + b_{\tau i} =0, \quad \text{ and } \quad b_i \equiv a_i   \mod l.
 \]
 
Now we can simply take 
\[
	 \mu = \sum_{i \in \Iwhite} a_i \omega_i - \sum_{i \in \Iwhite} b_i \omega_i.
\]

Then since $\langle i, \mu \rangle \in l\mathbb{Z}$ for all $i \in \I$, we have $\mu \in X^*$. Since $b_i + b_{\tau i} = 0$ for all $i \in \Iwhite$ by definition, we have 
\[
	\mu - \theta(\mu) = \sum_{i \in \Iwhite} (a_i + a_{\tau i } + b_i + b_{\tau i}) \omega_i = \lambda - \theta(\lambda).
\]
The lemma follows.
\end{proof}

{\bf For the rest of paper, we assume  $\breve{X}^* = X^* \cap \breve{X}$.} Thanks to the previous lemma, the equality holds when the root datum $(Y,X, \dots )$ is simply connected. 
\begin{lem}
Let $\lambda \in X$ such that $\overline{\lambda} \in X^*_\imath$, then $\lambda \in X^*$.
\end{lem}

\begin{proof}
We have 
\[
\lambda = \mu+ \nu -\theta(\nu), \text{ with } \mu, \nu \in X^*.
\]
Then obviously we have $\lambda \in X^*$. 
\end{proof}

\subsection{The $\imath$quantum group $\U^{*,\imath}$}
We now define the $\imath$quantum group associated with the root datum $(X^*, Y^*,\dots)$ and the pair $(\Iblack, \tau)$.
\begin{definition}
The algebra $\U^{*,\imath}$, with parameters 
\begin{equation}
  \label{parameters}
\vs_{i}^* = \vs_i^{l^2} , \quad \kappa_i^* = \kappa_i^{l^2}, \qquad \text{ for }  i \in \I_{\circ},
\end{equation}
is the $\Qq$-subalgebra of $\U^*$ generated by the following elements:
\begin{align*}
F_{i}  &+ \vs_i^* \T_{w_{\bullet}} (E_{\tau i}) \widetilde{K}^{-1}_i 
+ \kappa^* _i \widetilde{K}^{-1}_{i} \, (i \in \I_{\circ}), 
 \\
& \quad K_{\mu} \,(\mu \in Y^{\imath,*}), \quad F_i \,(i \in \I_{\bullet}), \quad E_{i} \,(i \in \I_{\bullet}).
\end{align*}
\end{definition}

\begin{lem}
The parameters $\vs^*_i$ and $\kappa_i^*$ satisfy the following conditions: 
\begin{align*}
 \begin{split}
\kappa^*_i &=0 \; \text{ unless } \tau(i) =i, \langle i^*, j'{^*} \rangle^* = 0 \; \forall j \in \Iblack,
\\
&  
\qquad 
\text{ and } \langle k^*,i'^* \rangle^* \in 2\Z \; \forall k = \tau(k) \in \Iwhite \text{ such that } \langle k^*, j'^* \rangle^* = 0 \; \text{ for all } j \in \Iblack;
\end{split}
\\
%
%
\vs^*_{i} & =\vs^*_{{\tau i}} \text{ if }    i \circ \theta (i) =0;
\\
\vs^*_{{\tau i}} &=   (-1)^{ \langle 2(\rho_{\bullet}^*)^\vee,  i'^* \rangle^* } (v^*_i)^{-\langle i^*, 2\rho^*_{\bullet}+\wb\tau i'^* \rangle^*} \overline{\vs^*_{i} }. 
\end{align*}
\end{lem}
\begin{proof}
The lemma follows directly from Lemma~\ref{lem:sametype}, and the fact that $v_i^* = v_i^{l^2}$.
\end{proof}

So $(\U^*, \U^{\imath, *})$ is also a quantum symmetric pair as defined in Definition~\ref{def:Ui}. Hence results from \cite{BW18b,BW18c} applies.
Therefore, we can similarly define the following as before ($R$ is any commutative $\mA$-algebra)
\[
{}_\mA\Udot^{*,\imath}, {}_R\Udot^{*,\imath}, \Udot^{*,\imath,\wedge},  {}_\mA\Udot^{*,\imath,\wedge}, {}_R\Udot^{*,\imath,\wedge}.
\]
In particular, we have the following counterpart of Proposition~\ref{prop:emd}:
\begin{equation}\label{eq:emb*}
\imath: {}_R \Udot^{*,\imath,\wedge} \lhook\joinrel\longrightarrow  {}_R \Udot^{*,\wedge}.
\end{equation}

\begin{rem}\label{rem:even}
If we take $l$ to be even, then the pair $(\Iblack, \tau)$ may not be admissible with respect to the root datum $(X^*, Y^*,\dots)$. Let us illustrate this phenomenon in the following example. 

Let $(\I = \{\alpha_1,\alpha_2\}, \cdot)$ be the Cartan datum of type $B_2$, where $\alpha_2$ denotes the short root. Let $(X, Y, \dots)$ be any root datum of type $(\I, \cdot)$. We take $\Iblack = \{\alpha_2\}$ and $\tau = \id$. It follows that $(\Iblack,\tau)$ is admissible, which one can see from the Satake diagram. Let $l = 4$. Then $(\I = \{\alpha_1,\alpha_2\}, \circ)$ is actually of type $C_2$ with short root $\alpha_1$. It is easy to see that $(\Iblack,\tau)$ is not admissible anymore, which one can again observe from the Satake diagram. 
\end{rem}
%
%
 
\section{Quantum Frobenius homomorphism}\label{sec:4}
Let $\mA\rq{}$ be the quotient of $\mA$ by the two-sided ideal generated by the $l$-th cyclotomic polynomial $f_{l} \in \mA$ through out this section. (Recall that $(f_1, f_2, f_3, \dots) = (v-1, v+1, v^2+v+1, \dots)$). We denote by $\phi : \mA \rightarrow \mA\rq{}$ the quotient map. 

Recall the Frobenius morphism in Theorem~\ref{thm:LuFr}. Thanks to the embeddings \eqref{eq:emb} and \eqref{eq:emb*}, we have
\[
\xymatrix{ {}_{\mA\rq{}}\Uidot  \ar[r]^-\imath  & {}_{\mA\rq{}}  {\Udot}^\wedge \ar[d]^-{\bold{Fr}}\\
	{}_{\mA\rq{}}\Udot^{*,\imath} \ar[r]^-\imath &  {}_{\mA\rq{}}  {\Udot}^{*,\wedge} }
.\]

Recall the parameters for the algebra $\Ui$ in \eqref{parameters}. {\bf For the rest of the paper, we shall only consider the case when the parameter $\kappa_i =0$.} We remark that this restriction is only relevant to the computation in \S\ref{subsec:AI1}. The goal of this section is to establish the following theorem:

\begin{thm}\label{thm:main}
	The quantum Frobenius morphism restricts to an $\mA\rq{}$-algebra homomorphism 
	\[
			\Fr: {}_{\mA\rq{}}\Uidot  \longrightarrow {}_{\mA\rq{}}\Udot^{*,\imath} .
	\]
\end{thm}

Recall Proposition~\ref{prop:idiv} that ${}_{\mA\rq{}}\Uidot $ is generated by the $\imath$-divided powers $B^{(a)}_{i,\zeta}$ ($i \in \I$) and $E^{(a)}_j 1_\zeta$ ($j \in \Iblack$) for $\zeta \in X_i$ and $a \ge 0$. It suffices to prove that 
\begin{enumerate}
\item\label{FrobEj} $ \Fr (\imath (E^{(a)}_j 1_{\zeta} )) = \Fr ( \sum_{\lambda \in X }E^{(a)}_j  1_\lambda) \in \text{Image of } \imath : {}_{\mA\rq{}}\Udot^{*,\imath} \rightarrow  {}_{\mA\rq{}}  {\Udot}^{*,\wedge}, $

\item\label{FrobBi} $ \Fr (\imath (B^{(a)}_{i,\zeta} )) = \Fr ( \sum_{\lambda \in X }B^{(a)}_{i,\zeta} 1_\lambda) \in \text{Image of } \imath : {}_{\mA\rq{}}\Udot^{*,\imath} \rightarrow  {}_{\mA\rq{}}  {\Udot}^{*,\wedge}.$
\end{enumerate}

The proof reduces to real rank one quantum symmetric pairs. The statement \eqref{FrobEj} is trivial.   We shall prove the statement \eqref{FrobBi} case by case for real rank one quantum symmetric pairs. We include the Satake diagram for real rank one quantum symmetric pairs for the readers' convenience. 

\begin{table}[h]
\caption{Satake diagrams of  
 symmetric pairs of real rank one}
\label{table:localSatake}
\begin{tabular}{| c | c || c | c |}
\hline
\begin{tikzpicture}[baseline=0]
\node at (0, -0.15) {AI$_1$};
\end{tikzpicture} 
& 
$	\begin{tikzpicture}[baseline=0]
		\node  at (0,0) {$\circ$};
		\node  at (0,-.3) {\small 1};
	\end{tikzpicture}
$ 
&
\begin{tikzpicture}[baseline=0]
\node at (0, -0.15) {AII$_3$};
\end{tikzpicture}
&
$	\begin{tikzpicture}[baseline=0]
		\node at (0,0) {$\bullet$};
		\draw (0.1, 0) to (0.4,0);
		\node  at (0.5,0) {$\circ$};
		\draw (0.6, 0) to (0.9,0);
		\node at (1,0) {$\bullet$};
		\node at (0,-.3) {\small 1};
		\node  at (0.5,-.3) {\small 2};
		\node at (1,-.3) {\small 3};
			\end{tikzpicture}	
$

\\
\hline
\begin{tikzpicture}[baseline=0]
\node at (0, -0.2) {AIII$_{11}$};
\end{tikzpicture}
 &
\begin{tikzpicture}[baseline = 6] 
		\node at (-0.5,0) {$\circ$};
		\node at (0.5,0) {$\circ$};
		\draw[<->] (-0.5, 0.2) to[out=45, in=180] (-0.15, 0.35) to (0.15, 0.35) to[out=0, in=135] (0.5, 0.2);
		\node at (-0.5,-0.3) {\small 1};
		\node at (0.5,-0.3) {\small 2};
	\end{tikzpicture}
	&
\begin{tikzpicture}[baseline=0]
\node at (0, -0.2) {AIV, n$\ge$2};
\end{tikzpicture} &
\begin{tikzpicture}	[baseline=6]
		\node at (-0.5,0) {$\circ$};
		\draw[-] (-0.4,0) to (-0.1, 0);
		\node  at (0,0) {$\bullet$};
		\node at (2,0) {$\bullet$};
		\node at (2.5,0) {$\circ$};
		\draw[-] (0.1, 0) to (0.5,0);
		\draw[dashed] (0.5,0) to (1.4,0);
		\draw[-] (1.6,0)  to (1.9,0);
		\draw[-] (2.1,0) to (2.4,0);
		\draw[<->] (-0.5, 0.2) to[out=45, in=180] (0, 0.35) to (2, 0.35) to[out=0, in=135] (2.5, 0.2);
		\node at (-0.5,-.3) {\small 1};
		\node  at (0,-.3) {\small 2};
		\node at (2.5,-.3) {\small n};
	\end{tikzpicture}
\\
\hline
\begin{tikzpicture}[baseline=0]
\node at (0, -0.2) {BII, n$\ge$ 2};
\end{tikzpicture} & 
    	\begin{tikzpicture}[baseline=0, scale=1.5]
		\node at (1.05,0) {$\circ$};
		\node at (1.5,0) {$\bullet$};
		\draw[-] (1.1,0)  to (1.4,0);
		\draw[-] (1.4,0) to (1.9, 0);
		\draw[dashed] (1.9,0) to (2.7,0);
		\draw[-] (2.7,0) to (2.9, 0);
		\node at (3,0) {$\bullet$};
		\draw[-implies, double equal sign distance]  (3.1, 0) to (3.7, 0);
		\node at (3.8,0) {$\bullet$};
		\node at (1,-.2) {\small 1};
		\node at (1.5,-.2) {\small 2};
		\node at (3.8,-.2) {\small n};
	\end{tikzpicture}	
&
\begin{tikzpicture}[baseline=0]
\node at (0, -0.15) {CII, n$\ge$3};
\end{tikzpicture}  
& 
		\begin{tikzpicture}[baseline=6]
		\draw (0.6, 0.15) to (0.9, 0.15);
		\node  at (0.5,0.15) {$\bullet$};
		\node at (1,0.15) {$\circ$};
		\node at (1.5,0.15) {$\bullet$};
		\draw[-] (1.1,0.15)  to (1.4,0.15);
		\draw[-] (1.4,0.15) to (1.9, 0.15);
		\draw (1.9, 0.15) to (2.1, 0.15);
		\draw[dashed] (2.1,0.15) to (2.7,0.15);
		\draw[-] (2.7,0.15) to (2.9, 0.15);
		\node at (3,0.15) {$\bullet$};
		\draw[implies-, double equal sign distance]  (3.1, 0.15) to (3.7, 0.15);
		\node at (3.8,0.15) {$\bullet$};
		\node  at (0.5,-0.15) {\small 1};
		\node at (1,-0.15) {\small 2};
		\node at (3.8,-0.15) {\small n};
	\end{tikzpicture}		
\\
\hline
\begin{tikzpicture}[baseline=0]
\node at (0, -0.05) {DII, n$\ge$4};
\end{tikzpicture}&
	\begin{tikzpicture}[baseline=0]
		\node at (1,0) {$\circ$};
		\node at (1.5,0) {$\bullet$};
		\draw[-] (1.1,0)  to (1.4,0);
		\draw[-] (1.4,0) to (1.9, 0);
		\draw[dashed] (1.9,0) to (2.7,0);
		\draw[-] (2.7,0) to (2.9, 0);
		\node at (3,0) {$\bullet$};
		\node at (3.5, 0.4) {$\bullet$};
		\node at (3.5, -0.4) {$\bullet$};
		\draw (3.1, 0.1) to (3.4, 0.35);
		\draw (3.1, -0.1) to (3.4, -0.35);
		\node at (1,-.3) {\small 1};
		\node at (1.5,-.3) {\small 2};
		\node at (3.5, 0.7) {\small n-1};
		\node at (3.5, -0.6) {\small n};
	\end{tikzpicture}		
&
\begin{tikzpicture}[baseline=0]
\node at (0, -0.2) {FII};
\end{tikzpicture}&
\begin{tikzpicture}[baseline=0][scale=1.5]
	\node at (0,0) {$\bullet$};
	\draw (0.1, 0) to (0.4,0);
	\node at (0.5,0) {$\bullet$};
	\draw[-implies, double equal sign distance]  (0.6, 0) to (1.2,0);
	\node at (1.3,0) {$\bullet$};
	\draw (1.4, 0) to (1.7,0);
	\node at (1.8,0) {$\circ$};
	\node at (0,-.3) {\small 1};
	\node at (0.5,-.3) {\small 2};
	\node at (1.3,-.3) {\small 3};
	\node at (1.8,-.3) {\small 4};
\end{tikzpicture}
\\
\hline
\end{tabular}
\newline
\smallskip
\end{table}

\subsection{}Let $i\in \Iwhite$ be such that $\tau (i)\neq i$. This includes type AIII$_{11}$ and type AIV.
We have $\kappa_i=0$, and $B_i = F_i +\vs_i \T_{w_{\bullet}} (E_{\tau i})  \tK^{-1}_i $. We write  
\[
Y_i : = \vs_i \T_{w_{\bullet}} (E_{\tau i})  \tK^{-1}_i, \qquad Y_i^{(n)} = Y_i^n / [n]^!_i.
\]

We have $F_i Y_i -v_i^{-2} Y_i F_i =[F_i,  Y_i] \tK^{-1}_i =0$.  Following \cite[\S5.5.1]{BW18c}, we define
\begin{equation}
B_{i,\zeta}^{(n)} = \frac{B_i^n}{[n]_i!} \one_{\zeta}
= \sum_{a=0}^n  v_i^{-a(n-a)}  Y_i^{(a)}F_i^{(n-a)} \one_\zeta \in {}_\mA \Uidot. 
\end{equation}

\begin{prop}\label{prop:ineqtaui}
We have that 
\begin{align*}
	\Fr:  {}_{\mA\rq{}}  {\Udot}^\wedge &\longrightarrow  {}_{\mA\rq{}}  {\Udot}^{*,\wedge},\\
 \sum_{\lambda \in X }B^{(n)}_{i,\zeta} 1_\lambda &\mapsto 
\begin{cases}
 \sum_{\lambda \in X^* } B^{(n/l)}_{i, \zeta} 1_\lambda, &\text{if } n \in l\Z, \zeta \in X^*_\imath;\\
0, &\text{otherwise.}
\end{cases} 
\end{align*}
In other words, we have 
\begin{align*}
	\Fr:  {}_{\mA\rq{}}\Uidot  &\longrightarrow  {}_{\mA\rq{}}\Udot^{*,\imath},\\
B^{(n)}_{i,\zeta}  &\mapsto \begin{cases}
 B^{(n/l)}_{i, \zeta}, &\text{if } n \in l\Z, \zeta \in X^*_\imath;\\
0, &\text{otherwise.}
\end{cases} 
\end{align*}

\end{prop}

\begin{proof}It suffices to check that 
\[
\bold{Fr} (B^{(n)}_{i,\zeta} 1_\lambda ) =  B^{(n/l)}_{i, \zeta} 1_\lambda, \text{ for } \zeta \in X_\imath^*, \overline{\lambda} = \zeta.
\]
We check only the cases where $n = kl \in l\mathbb{Z}$. The other cases are entirely similar. 

We can use the quantum binomial formula (\cite[\S 1.3.5]{Lu}) to write, for $\lambda \in X^{*}$ with $\overline{\lambda} =\zeta$,
\begin{align*}
\bold{Fr}(B^{(n)}_{i, \zeta} \one_{\lambda}) &= \bold{Fr}((\sum_{a=0}^{n}  v_i^{-a(n-a)}  Y_i^{(a)}F_i^{(n-a)})\one_\lambda)\\
&\stackrel{\heartsuit}{=} (\sum_{a=0}^{k}  (v_i^*)^{-a(n-a)}  Y_i^{(a)}F_i^{(n-a)})\one_\lambda\\
&= B_{i,\zeta}^{(k)}\one_{\la},
\end{align*}
where $\heartsuit$ follows from Theorem~\ref{thm:LuFr}, and the fact that $\vs_i^{l^2_i} =\vs_i^*$.
\end{proof}
\subsection{}
 Let $i \in \Iwhite$ such that  $\tau (i) = i \neq \wb(i)$. This includes the types: BII, DII, AII$_3$, CII, and FII. We have $\kappa_i=0$, and $B_i = F_i +\vs_i \T_{w_{\bullet}} (E_{\tau i})  \tK^{-1}_i $. 

We write  \[
 Y_i =  \vs_i \T_{w_{\bullet}}(E_i) \widetilde{K}_i^{-1}, \quad b^{(n)}_i = \sum_{a=0}^n v_i^{-a(n-a)} Y_i^{(a)} F_i^{(n-a)}.
\]
Following (\cite[(5.12)]{BW18c}), we define
\begin{equation}\label{eq:Bi}
B^{(n)}_{i, \zeta} \one_{\lambda} =  b_i^{(n)}\one_{\lambda} + \frac{v}{v-v^{-1}} \sum_{k \ge 1} v_i^{\frac{k (k+1)}{2}} \mathfrak{Z}_{i}^{(k)}  b_i^{(n-2k)}\one_{\lambda},
\end{equation}
where 
\[
\mathfrak{Z}_{i}^{(n)} = -v_{i}^{\frac{1}{2}n(n-1)}(\sum_{a=0}^{n-1}v_{i}^{-2n^{2} + 2na -\frac{1}{2}a(a-1)}Y_{i}^{(n-a)}F_{i}^{(n-a)}\mathfrak{Z}_{i}^{(a)} - F_{i}^{(n)}Y_{i}^{(n)}).
\]


\begin{lem}\label{lem:bi}
We have that
\begin{align*}
	\Fr:  {}_{\mA\rq{}}  {\Udot}^\wedge &\longrightarrow  {}_{\mA\rq{}}  {\Udot}^{*,\wedge},\\
 \sum_{\lambda \in X }b^{(n)}_{i} 1_\lambda &\mapsto  
\begin{cases}
 \sum_{\lambda \in X^* } b^{(n/l)}_{i} 1_\lambda, &\text{if } n \in l\Z;\\
0, &\text{otherwise.}
\end{cases} 
\end{align*}
In other words, we have
\begin{align*}
	\Fr:  {}_{\mA\rq{}}\Uidot  &\longrightarrow  {}_{\mA\rq{}}\Udot^{*,\imath},\\
b^{(n)}_{i} 1_\zeta &\mapsto
\begin{cases}
b^{(n/l)}_{i} 1_\zeta, &\text{if } n \in l\Z, \zeta \in X^*_\imath;\\
0, &\text{otherwise.}
\end{cases} 
\end{align*}
\end{lem}
\begin{proof}
This is the same computation as Proposition~\ref{prop:ineqtaui}.
\end{proof}

\begin{prop}
We have that
\begin{align*}
	\Fr:  {}_{\mA\rq{}}  {\Udot}^\wedge &\longrightarrow  {}_{\mA\rq{}}  {\Udot}^{*,\wedge},\\
\sum_{\lambda \in X } B^{(n)}_{i, \zeta} 1_\lambda &\mapsto
\begin{cases}
\sum_{\lambda \in X^*} b^{(n/l)}_{i} 1_\lambda, &\text{if } n \in l\Z, \zeta \in X^*_\imath;\\
0, &\text{otherwise.}
\end{cases} 
\end{align*}
In other words, we have
\begin{align*}
	\Fr:  {}_{\mA\rq{}}\Uidot  &\longrightarrow  {}_{\mA\rq{}}\Udot^{*,\imath},\\
B^{(n)}_{i} 1_\zeta &\mapsto
\begin{cases}
b^{(n/l)}_{i} 1_\zeta, &\text{if } n \in l\Z, \zeta \in X^*_\imath;\\
0, &\text{otherwise.}
\end{cases} 
\end{align*}\end{prop}

\begin{proof}
We first have  
\[
\bold{Fr}(\mathfrak{Z}_{i}^{(n)} \one_{\lambda}) = 0, \qquad \text{ for } 1 \leq n \leq l-1, \lambda \in X.
\]
 Moreover,
\[
\bold{Fr}(\mathfrak{Z}_{i}^{(l)}  \one_{\lambda}) = -  \bold{Fr}(Y_i^{(l)} F_i^{(l)}  - F_i^{(l)}  Y_i^{(l)})  \one_{\lambda}= \varsigma_{i}^{l}[F_{i}, \T_{w_{\bullet}}(E_i)]  \one_{\lambda} = 0, \text{ for any }\lambda \in X. 
\]

The last equality follows from case by case computation. It suffices to notice that in none of these cases is $w_{\bullet}(i\rq{}) - i\rq{}$ a root.

For $1 \leq n \leq l - 1,$
\[
0 = \bold{Fr}(\mathfrak{Z}_{i}^{(kl)}  \one_{\mu\rq{}})\bold{Fr}(\mathfrak{Z}_{i}^{(n)} \one_{\mu} ) = \left[ \begin{array}{cc|r} kl + n \\ n \end{array} \right]_{i}\bold{Fr}(\mathfrak{Z}_{i}^{(kl+n)} \one_{\mu})
\]
and 
\[
0 = \bold{Fr}(\mathfrak{Z}_{i}^{(l)}  )^{k}  \one_{\mu}= k!\bold{Fr}(\mathfrak{Z}_{i})^{(kl)} \one_{\mu}.
\]
Therefore we have 
\[
\bold{Fr}( \mathfrak{Z}_{i}^{(n)} \one_{\mu}) =  0 \quad \mbox{for $n \geq 1$}.
\]
The proposition follows from \eqref{eq:Bi} and Lemma~\ref{lem:bi}.
\end{proof}

\subsection{}\label{subsec:AI1}
In this final section, we consider the case $\tau i = i = \wb i$. This is of type AI$_1$. 
In this case, we have $B_i = F_i + \vs_i E_i \widetilde{K}_i^{-1} $.  Following the definition of $\imath$divided powers in this case given in \cite{BW18a}, the precise formula for $B_{i,\zeta}^{(n)}$ has been obtained in \cite{BeW18} (recall the parameter $\kappa_i = 0$).

Let $\lambda \in X$ such that $\overline{\lambda} = \zeta$. The parity of $\langle i, \lambda \rangle \in \mathbb{Z}$ depends only on $\zeta$, but not on $\lambda$. We shall simply call it the parity of $\langle i, \zeta \rangle$. The computation divides into two cases depends on the parity of $\langle i, \zeta \rangle$. We shall focus on the case when  $\langle i, \zeta \rangle$ is even. The odd case is entirely similar.
%

\begin{prop}\cite[Proposition~2.8]{BeW18}  Let $\langle i, \zeta \rangle$ be even. Let $m \ge 1$, and $\lambda \in X$ such that $\overline{\lambda} = \zeta$. We have 
\begin{equation}\label{eq:2m}
	\begin{split}
B_{i, \zeta}^{(2m)} 1_\lambda  
= & \sum_{c=0}^{m} \sum_{a=0}^{2m-2c} v_i^{2(a+c)(m-a-\langle i, \lambda \rangle/2) - 2ac - \binom{2c+1}{2}} \\
&  \cdot\left[ \begin{array}{cc|r} m-c-a -\langle i, \lambda \rangle/2\\ c \end{array} \right]_{v_i^{2}}E^{(a)}_i F_i^{(2m-2c-a)} 1_\lambda,
\end{split}
\end{equation}

\begin{equation}\label{eq:2m-1}
	\begin{split}
B_{i, \zeta}^{(2m-1)} 1_\lambda = &\sum_{c=0}^{m-1} \sum_{a=0}^{2m-1-2c} v_i^{2(a+c)(m-a-\langle i, \lambda \rangle/2) - 2ac -a- \binom{2c+1}{2}} \, \cdot \\
\quad &\left[ \begin{array}{cc|r} m-c-a -\langle i, \lambda \rangle/2 -1\\ c \end{array} \right]_{v_i^{2}}E^{(a)}_i F_i^{(2m-1-2c-a)} 1_\lambda.
	\end{split}
\end{equation}

\end{prop}


\begin{prop}\label{prop:teven}
	 Let  $\langle i, \zeta \rangle$ be even. We have, via restriction, 
	 \begin{align*}
	 	\Fr:   {}_{\mA\rq{}}  {\Udot}^\wedge &\longrightarrow  {}_{\mA\rq{}}  {\Udot}^{*,\wedge},\\
\sum_{\lambda \in X}B_{i, \zeta}^{(n)} 1_\lambda&\mapsto 
	\left\{ \begin{array}{rcl}
							&\sum_{\lambda \in X^*}B_{i, \zeta}^{(n/l)} 1_\lambda, &\mbox{for $n \in l\mathbb{Z}$ and $\zeta \in X_\imath^{*}$}; \vspace{.2cm}\\ 
							&\sum_{\lambda \in X^*} \left[ \begin{array}{cc|r} (l-1)/2 \\ b/2 \end{array} \right]_{v_i^{2}}B_{i, \zeta}^{(k)} 1_\lambda, &\begin{aligned}&\text{for $n = kl + b$, $k$ odd,}\\ &\text{$b$ even, $0 < b < l$, $\zeta \in X_\imath^{*}$}; \end{aligned}\\
							&0,  &\mbox{otherwise} .
					\end{array}\right.
	\end{align*}

	 In other words, we have
	 \begin{align*}
	 	\Fr:  {}_{\mA\rq{}}\Uidot  &\longrightarrow  {}_{\mA\rq{}}\Udot^{*,\imath},\\
B_{i, \zeta}^{(n)} &\mapsto 
	\left\{ \begin{array}{rcl}
							&B_{i, \zeta}^{(n/l)}, &\mbox{for $n \in l\mathbb{Z}$ and $\zeta \in X_\imath^{*}$}; \vspace{.2cm}\\ 
							&\left[ \begin{array}{cc|r} (l-1)/2 \\ b/2 \end{array} \right]_{v_i^{2}}B_{i, \zeta}^{(k)}, &\begin{aligned}&\text{for $n = kl + b$, $k$ odd,}\\ &\text{$b$ even, $0 < b < l$, $\zeta \in X^{*}_\imath$}; \end{aligned}\\
							&0,  &\mbox{otherwise} .
					\end{array}\right.
	\end{align*}
\end{prop}

\begin{proof}
Let $\lambda \in X^*$ such that $\overline{\lambda} = \zeta$. We divide the computation into several cases.

(1) We first consider the case when $n =2m$ is even, where either  $ 0<n=2m<l$ or $n=2m=kl$. Recall \eqref{eq:2m} for the expression of $B_{i, \zeta}^{(n)} 1_\lambda$.

	(1.1) We assume $ 0<n=2m<l$. Then it follows from direct computation and  \cite[Lemma~34.1.2]{Lu} that 
	\[
	\Fr (B_{i, \zeta}^{(n)} 1_\lambda) = 0.
	\]

	(1.2) We assume $n=2m=kl$. Note that 
\[
\Fr(E_i^{(a)} F_i^{(2m-2c-a)}) = 
\begin{cases}
	E_i^{(a/l)} F_i^{((2m-2c-a)/l)}, &\text{if } l \vert a \text{ and } l \vert c;\\
	0, &\text{otherwise}.
\end{cases}
\]
However, if $l \vert a$ and $l \vert c$, we must have $v_i^{2(a+c)(m-a-\langle i, \lambda \rangle/2) - 2ac - \binom{2c+1}{2}} =1$.

  Then, we have 
\begin{align*}
	&\bold{Fr}(B_{i, \zeta}^{(2kl)} 1_\lambda ) \\
	= &\sum_{c=0}^{k} \sum_{a=0}^{2k-2c} \left[ \begin{array}{cc|r} kl-cl-al - \langle i, \lambda \rangle/2\\ cl \end{array} \right]_{v_i^{2}}E^{(a)}_i F_i^{(2k-2c-a)} 1_\lambda\\
	= &\sum_{c=0}^{k} \sum_{a=0}^{2k-2c} \left( \begin{array}{cc|r} k-c-a - \langle i, \lambda \rangle^*/2\\ c \end{array} \right) E^{(a)}_i F_i^{(2k-2c-a)} 1_\lambda\\
	= & B_{i, \zeta}^{(2k)} 1_\lambda.
\end{align*}

(2) We then consider the case when $n =2m-1$ is odd, where either $0<n=2m-1<l$ or $n=2m-1 = (2k-1)l$. Recall \eqref{eq:2m-1} for the expression of $B_{i, \zeta}^{(n)} 1_\lambda$ in this case.

	(2.1) We assume $n = 2m-1 < l$. Then by direct computation, we must have 
	\[
	\Fr(B_{i, \zeta}^{(n)} 1_\lambda ) = 0 .
	\]

 (2.2) We assume $n =2m-1= (2k-1)l$ for $2k-1 >0$. Then we have 
 \[
 	m-1 = (k-1)l + (l-1)/2, \quad\text{with } 0< (l-1)/2 < l.
 \]
 Then 
 \begin{align*}
 B_{i, \zeta}^{(n)} 1_\lambda = &\sum_{c=0}^{(k-1)l + ({l-1})/{2}} \sum_{a=0}^{(2k-1) l-2c} v_i^{2(a+c)(m-a-\langle i, \lambda \rangle/2) - 2ac -a- \binom{2c+1}{2}} \, \cdot \\
\quad &\left[ \begin{array}{cc|r} m-c-a -\langle i, \lambda \rangle/2 -1\\ c \end{array} \right]_{v_i^{2}}E^{(a)}_i F_i^{(2m-1-2c-a)} 1_\lambda.
 \end{align*}
Note that 
\[
\Fr(E_i^{(a)} F_i^{(2m-1-2c-a)}) = 
\begin{cases}
	E_i^{(a/l)} F_i^{((2m-1-2c-a)/l)}, &\text{if } l \vert a \text{ and } l \vert c;\\
	0, &\text{otherwise}.
\end{cases}
\]
However, if $l \vert a$ and $l \vert c$, we must have $v_i^{2(a+c)(m-a-\langle i, \lambda \rangle /2 ) - 2ac -a- \binom{2c+1}{2}} =1$. 

Therefore
\begin{align*}
	&\bold{Fr}(B_{i, \zeta}^{(n)} 1_\lambda ) \\
	= &\sum_{c=0}^{(k-1)} \sum_{a=0}^{2k-1-2c} \left[ \begin{array}{cc|r}(k-1)l + ({l-1})/{2}-cl-al - \langle i, \lambda \rangle/2\\ cl \end{array} \right]_{v_i^{2}}E^{(a)}_i F_i^{(2k-1-2c-a)} 1_\lambda\\
	\stackrel{\spadesuit}{=} &\sum_{c=0}^{(k-1)} \sum_{a=0}^{2k {-1}-2c} \left( \begin{array}{cc|r} k-1-c-a - \langle i, \lambda \rangle^*/2\\ c \end{array} \right) E^{(a)}_i F_i^{(2k-1-2c-a)} 1_\lambda\\
{=} & {B_{i, \zeta}^{(2k-1)} 1_\lambda},
\end{align*}
where $\spadesuit$ follows from \cite[Lemma~34.1.2]{Lu}.

(3) At last we consider the case where $n = kl +b$ with $0<b<l$. 
Note that since $\mA\rq{}$ is an integral domain, it suffices to perform the computation in the field of fractions. 

Recall the following induction formula from \cite[(2.5)]{BeW18}: 
\begin{align*}
	B_i B_{i, \zeta}^{(kl+c)} &= [kl+c+1]_{v_i} B_{i, \zeta}^{(kl+c+1)}, \quad &\text{if } kl+c \text{ is odd; }\\
	B_i B_{i, \zeta}^{(kl+c)} &= [kl+c+1]_{v_i} B_{i, \zeta}^{(kl+c+1)} +  [kl+c]_{v_i} B_{i, \zeta}^{(kl+c-1)}, \quad &\text{if } kl+c \text{ is even. }
\end{align*}
Recall $v_i^{l} =1$. Note that $[kl]_{v_i}=0$ and $[kl+b]_{v_i}  = [b]_{v_i}\neq 0$ for all $0<b<l$.

 (3.1) When $k$ (hence also $kl$) is odd, we have 
 \[ 
 \Fr (B_{i, \zeta}^{(kl+b)} 1_\lambda) =  \begin{cases} (-1) \frac{[b-1]_{v_i}}{[b]_{v_i}} \Fr (B_{i, \zeta}^{(kl+b-2)} 1_\lambda), &\text{if } b \text{ is even;}\\
  0, &\text{if } b \text{ is odd.}
 \end{cases} 
\]
 Therefore we have 

\begin{align*}
 \Fr (B_{i, \zeta}^{(kl+b)} 1_\lambda) &= (-1)^{b/2 } \frac{[b-1]_{v_i}}{[b]_{v_i}} \frac{[b-3]_{v_i}}{[b-2]_{v_i}} \cdots \frac{[1]_{v_i}}{[2]_{v_i}} \Fr (B_{i, \zeta}^{(kl)} 1_\lambda) \\
 & = \frac{[l-b+1]_{v_i}}{[b]_{v_i}} \frac{[l-b+3]_{v_i}}{[b-2]_{v_i}} \cdots \frac{[l-1]_{v_i}}{[2]_{v_i}} \Fr (B_{i, \zeta}^{(kl)} 1_\lambda) \\
 & = \left[ \begin{array}{cc|r} (l-1)/2 \\ b/2 \end{array} \right]_{v_i^{2}}  \Fr (B_{i, \zeta}^{(kl)} 1_\lambda).
\end{align*}



(3.2)  When $k$ (hence also $kl$) is even, note that 
\[
B_i B_{i, \zeta}^{(kl)} 
=[kl+1]_{v_i} B_{i, \zeta}^{(kl+1)}, \qquad \text{ since } [kl]_{v_i}=0.
\]

Therefore by a similar computation as above, we have 
\[
 \Fr (B_{i, \zeta}^{(kl+b)} 1_\lambda) =  0, \quad \text{ for all } 0<b<l.
\]
%
\end{proof}

\begin{prop}

Let $\langle i,\zeta \rangle$ be odd. We have 
 \begin{align*}
	 	\Fr:   {}_{\mA\rq{}}  {\Udot}^\wedge &\longrightarrow  {}_{\mA\rq{}}  {\Udot}^{*,\wedge},\\
\sum_{\lambda \in X}B_{i, \zeta}^{(n)} 1_\lambda&\mapsto 
	\left\{ \begin{array}{rcl}
							&\sum_{\lambda \in X^*}B_{i, \zeta}^{(n/l)} 1_\lambda, &\mbox{for $n \in l\mathbb{Z}$ and $\zeta \in X_\imath^{*}$}; \vspace{.2cm}\\ 
							&\sum_{\lambda \in X^*} \left[ \begin{array}{cc|r} (l-1)/2 \\ b/2 \end{array} \right]_{v_i^{2}}B_{i, \zeta}^{(k)} 1_\lambda, &\begin{aligned}&\text{for $n = kl + b$, $k$ even,}\\ &\text{$b$ even, $0 < b < l$, $\zeta \in X^{*}$}; \end{aligned}\\
							&0,  &\mbox{otherwise} .
					\end{array}\right.
	\end{align*}

In other words, we have 
 \begin{align*}
 	\Fr:  {}_{\mA\rq{}}\Uidot  &\longrightarrow  {}_{\mA\rq{}}\Udot^{*,\imath},\\
 B_{i, \zeta}^{(n)} &\mapsto \left\{  \begin{array}{rcl}
							&B_{i, \zeta}^{(n/l)}, &\mbox{for $n \in l\mathbb{Z}$ and $\la \in X_\imath^{*}$}; \vspace{.2cm} \\ 
							&\left[ \begin{array}{cc|r} (l-1)/2 \\ b/2 \end{array} \right]_{v_i^{2}}B_{i, \zeta}^{(k)},  &\!\begin{aligned}&\text{for $n = kl + b$, $k$ even,}\\ &\text{$b$ even, $0 < b < l$, $\zeta \in X_\imath^{*}$;}\end{aligned} \\
							&0,  &\mbox{otherwise}. 
					\end{array}\right.
 \end{align*}

\end{prop}
\begin{proof}
The computation is entirely simialr to that of Proposition~\ref{prop:teven}. We shall omit the details here. 
\end{proof}
\section{Small quantum symmetric pairs}\label{sec:5}
\subsection{} 
Let ${}_{\mA'}\udot$ be the $\mA'$-subalgebra of $_{\mA'} \Udot$ generated by $E^{(n)}_i 1_\lambda$, $F^{(n)}_i 1_\lambda$ for various $i \in \I$, various $n$ such that $0 \le n <l$ and various $\lambda \in X$. Let ${}_{\mA'}\dot{\mathfrak{p}} = {}_{\mA'}\dot{\mathfrak{p}}_{\I_\bullet}$ be the $\mA'$-subalgebra of $_{\mA'} \Udot$ generated by $E^{(n)}_i 1_\lambda$, $F^{(n)}_j 1_\lambda$ for various $i \in \I_\bullet, j \in \I$, various $n$ such that $0 \le n <l$ and various $\lambda \in X$.  

\begin{definition}
Let ${}_{\mA'}\dot{\mathfrak{u}}^\imath$ be the $\mA'$-subalgebra of ${}_{\mA'}\Uidot$ generated by 
$
 B^{(n)}_{i,\zeta}$, $E_{j}^{(n)} 1_\zeta$ for various $i \in \I$, $j \in \Iblack$,  $\zeta \in X_\imath$ and $n$ such that $0 \le n <l$. For any $\mA'$-commutative ring $R$, we define 
${}_{R}\dot{\mathfrak{u}}^\imath = R \otimes_{\mA'} {}_{\mA'}\dot{\mathfrak{u}}^\imath$.

\end{definition}

 We call $({}_{\mA'}\dot{\mathfrak{u}}^\imath, {}_{\mA'}\udot)$ the small quantum symmetric pair. It is clear that ${}_{\mA'}\dot{\mathfrak{u}}^\imath$ is a ``coideal subalgebra" of  ${}_{\mA'}\udot$, that is, we have (via restriction)
 \[
 	\Delta: {}_{\mA'}\dot{\mathfrak{u}}^\imath \longrightarrow \prod_{\zeta \in X_\imath^*, \lambda \in X^*} {}_{\mA'}\dot{\mathfrak{u}}^\imath 1_\zeta \otimes_{\mA'} {}_{\mA'}\udot 1_\lambda.
 \]

\begin{rem}
Here we abuse the terminology ``coideal subalgebra", even though ${}_{\mA'}\dot{\mathfrak{u}}^\imath$ is not a subalgebra of ${}_{\mA'}\udot$. 
\end{rem}

\begin{thm}\label{thm:main2}
For any $\zeta \in X_\imath$, ${}_{\mA'}\dot{\mathfrak{u}}^\imath 1_\zeta$ is a free $\mA'$-module with rank $l^{\vert \Phi^{+}_\bullet \vert+ \vert \Phi^+ \vert}$.
\end{thm}

\begin{proof}
We have the following linear isomorphism via a base change from \cite[Corollary~6.20]{BW18b} 
\[ 
p_{\imath,\lambda}: \xymatrix{{}_{\mA'} \Uidot 1_\zeta \ar[r]^{\cong}& {}_{\mA'} \dot{\bold{P}}1_\lambda}, \qquad \overline{\lambda} =\zeta, \lambda \in X.
\]
Via restriction, we have the map
\[
p_{\imath,\lambda}: {}_{\mA'} \dot{\mathfrak{u}}^\imath 1_\zeta \longrightarrow {}_{\mA'}\dot{\mathfrak{p}}1_\lambda,
\]
whose surjectivity can be obtained similar to \cite[Corollary~6.20]{BW18b}. We know the image lies in  ${}_{\mA'}\dot{\mathfrak{p}}1_\lambda$ thanks to the precise formulas of $
 B^{(n)}_{i,\zeta}$ obtained in \cite{BW18b, BeW18}.

The theorem follows immediately from the fact that ${}_{\mA'}\dot{\mathfrak{p}}1_\lambda$ is a free $\mA'$-module with rank $l^{\vert \Phi^{+}_\bullet \vert+ \vert \Phi^+ \vert}$ thanks to \cite[Theorem~8.3]{Lu90b}.
\end{proof}

%
%
%
%

\bibliography{Ref}
\bibliographystyle{amsalpha}

\end{document}